\documentclass[12pt]{article}

\usepackage{amsthm, amsmath, amssymb}

\numberwithin{equation}{section}

\theoremstyle{plain}	     
\newtheorem{thm}{Theorem}[section] 
\newtheorem{cor}[thm]{Corollary}

\newtheorem{prop}[thm]{Proposition}
     
\theoremstyle{definition}

\theoremstyle{remark} 
\newtheorem{rem}[thm]{Remark}
	
\newcommand{\disp}{\displaystyle}

\begin{document}
\title{Legendre-type relations for generalized complete elliptic integrals
\footnote{This work was supported by JSPS KAKENHI Grant Number 24540218.}}
\author{Shingo Takeuchi\\
Department of Mathematical Sciences\\
Shibaura Institute of Technology
\thanks{307 Fukasaku, Minuma-ku,
Saitama-shi, Saitama 337-8570, Japan. \endgraf
{\it E-mail address\/}: shingo@shibaura-it.ac.jp \endgraf
{\it 2010 Mathematics Subject Classification.} 
33E05, 33C75, 34L10}}
\date{}

\maketitle

\begin{abstract}
Legendre's relation for the complete elliptic integrals of the 
first and second kinds is generalized. The proof depends on an application 
of the generalized trigonometric functions and is alternative to
the proof for Elliott's identity.
\end{abstract}

\textbf{Keywords:} 
Legendre's relation,
complete elliptic integrals,
generalized trigonometric functions,
Elliott's identity

%%%%%%%%%%%%%%%%%%%%%%%%%%%%%%%%%%%%%%%%%%%%%%%%%%

\section{Introduction}

Let $k \in [0,1)$. The complete elliptic integrals of the first kind
$$K(k)=\int_0^1 \frac{dt}{\sqrt{(1-t^2)(1-k^2t^2)}}$$
and of the second kind
$$E(k)=\int_0^1 \sqrt{\frac{1-k^2t^2}{1-t^2}}\,dt$$
play important roles in classical analysis. 
In this paper, we consider generalizations of $K(k)$ and $E(k)$ as
$$K_{p,q,r}(k):=\int_0^1 \frac{dt}{(1-t^q)^{1/p} (1-k^qt^q)^{1/r}}$$
and
$$E_{p,q,r}(k):=\int_0^1 \frac{(1-k^qt^q)^{1/r^*}}{(1-t^q)^{1/p}}\,dt,$$
where $p \in \mathbb{P}^*:=(-\infty,0) \cup (1,\infty],\ 
q,\ r \in (1,\infty)$ and $1/s+1/s^*=1$. 
For $p=\infty$ we regard $K_{p,q,r}$ and $E_{p,q,r}$ as 
$$K_{\infty,q,r}(k):=\int_0^1\frac{dt}{(1-k^qt^q)^{1/r}},\quad
E_{\infty,q,r}(k):=\int_0^1(1-k^qt^q)^{1/r^*}\,dt.$$
Under the convention that $1/\infty=0$ and $1/0=\infty$,
we should note that
$s \in \mathbb{P}^*$ if and only if $s^* \in (0,\infty)$,
particularly, $\infty^*=1$.
In case $p=q=r=2$, $K_{p,q,r}(k)$ and $E_{p,q,r}(k)$ are reduced to 
the classical $K(k)$ and $E(k)$, respectively.

There is a lot of literature about 
the generalized complete elliptic integrals.
$K_{p,q,p}$ is introduced
in \cite{T} with a generalization of the Jacobian elliptic function
with a period of $4K_{p,q,p}$ to study a bifurcation problem of a bistable 
reaction-diffusion equation involving $p$-Laplacian.
%In that work, the constant $\pi_{p/2,p}$ appears for $p>2$, 
%which is equal to $\pi_{s,q}$ with $s=p/2$ defined above.
Relationship between $K_{p,q,p}$ and $E_{p,q,p^*}$ has been observed
in \cite{BY,YH}.
Regarding $K_{p,q,p^*}$,
another generalization of Jacobian elliptic function with a period of $K_{p,q,p^*}$
is given and the basis properties for the family of these functions
are shown in \cite{T2}. Moreover, $K_{p,q,p^*}$ is also applied to a problem 
on Bhatia-Li's mean and a curious relation between $K_{p,q,p^*}$ and $E_{p,q,p^*}$
is given in \cite{KT}. 
%Concerning Legendre's relation, $K_{p,p,p^*}$ and $E_{p,p,p^*}$ satisfy
%have been applied
%in \cite{T3,T4} to generalize \eqref{eq:L} and 
%to establish computation formulas of $\pi_p$
%(only for $p=3$ and $4$), the generalized $\pi$ defined below.

It is well known that $K(k)$ and $E(k)$ satisfy the famous \textit{Legendre's relation}
(see, for example, \cite{AAR, BB3, Du}):
\begin{equation}
\label{eq:L}
E(k)K(k')+K(k)E(k')-K(k)K(k')=\frac{\pi}{2},
\end{equation}
where $k'=\sqrt{1-k^2}$.
Our purpose in the present paper 
is to generalize Legendre's relation \eqref{eq:L}
to the generalized complete elliptic integrals above.

To state the results, we will give some notations. 
For $p \in \mathbb{P}^*$ and $q \in (1,\infty)$, let 
$$\pi_{p,q}:=2\int_0^1 \frac{dt}{(1-t^q)^{1/p}}
=\frac{2}{q}B\left(\frac{1}{q},\frac{1}{p^*}\right),$$
where $B$ denotes the beta function.
In particular, $\pi_{\infty,q}=2$ for any $q \in (1,\infty)$.
We write $K_{p,q}:=K_{p,q,q^*},\ E_{p,q}:=E_{p,q,q^*}$ for $p \in \mathbb{P}^*$
and $q \in (1,\infty)$; $K_p:=K_{p,p,p^*},\ E_p:=E_{p,p,p^*},\ 
\pi_p:=\pi_{p,p}$ for $p \in (1,\infty)$.

\begin{thm}
\label{thm:pqrL}
Let $p \in \mathbb{P}^*,\ q,\ r \in (1,\infty)$ and $k \in (0,1)$. Then
\begin{multline}
\label{eq:pqrL}
E_{p,q,r^*}(k)K_{p,r,q^*}(k')
+K_{p,q,r^*}(k)E_{p,r,q^*}(k')\\
-K_{p,q,r^*}(k)K_{p,r,q^*}(k')
=\frac{\pi_{p,q}\pi_{s,r}}{4},
\end{multline}
where $k':=(1-k^q)^{1/r}$ and $1/s=1/p-1/q$.
\end{thm}

\begin{cor}[Case $q=r$]
\label{cor:q=r}
Let $p \in \mathbb{P}^*,\ q \in (1,\infty)$ and $k \in (0,1)$. Then
\begin{equation}
\label{eq:pqL}
E_{p,q}(k)K_{p,q}(k')
+K_{p,q}(k)E_{p,q}(k')\\
-K_{p,q}(k)K_{p,q}(k')
=\frac{\pi_{p,q}\pi_{s,q}}{4},
\end{equation}
where $k':=(1-k^q)^{1/q}$ and $1/s=1/p-1/q$.
\end{cor}

\begin{cor}[\cite{T3}, Case $p=q=r$]
Let $p \in (1,\infty)$ and $k \in (0,1)$. Then
\begin{equation}
\label{eq:pL}
E_{p}(k)K_{p}(k')
+K_{p}(k)E_{p}(k')\\
-K_{p}(k)K_{p}(k')
=\frac{\pi_{p}}{2},
\end{equation}
where $k':=(1-k^p)^{1/p}$.
\end{cor}

\begin{rem}
Using \eqref{eq:pL}, the author establishes computation formulas of $\pi_p$
for $p=3$ in \cite{T3}; for $p=4$ in \cite{T4}.
\end{rem}

In fact, \eqref{eq:pqrL} is equivalent to \textit{Elliott's identity}
\eqref{eq:elliott} below.
The advantage of our result lies in the facts that  
it is understandable without acknowledge of hypergeometric functions
and that its proof gives an alternative proof for Elliott's identity with straightforward calculations. 
%Our proof of Theorem \ref{thm:pqrL} is, however, 
%not completely satisfactory,
%because it depends upon prior knowledge of the identity.   

%%%%%%%%%%%%%%%%%%%%%%%%%%%%%%%%%%%%%%%

\section{Proof of Theorem \ref{thm:pqrL}}

The following property immediately follows from the definitions of $K_{p,q,r}$ and $E_{p,q,r}$.

\begin{prop}
\label{prop:behaviour}
Let $p \in \mathbb{P}^*,\ q,\ r \in (1,\infty)$.
Then, 
$K_{p,q,r}(k)$ is increasing on $[0,1)$ and
\begin{align*}
K_{p,q,r}(0)&=\frac{\pi_{p,q}}{2},\\
\lim_{k \to 1-0}K_{p,q,r}(k)&=
\begin{cases}
\infty & \mbox{if}\ 1/p+1/r \ge 1,\\
\pi_{u,q}/2\ (1/u=1/p+1/r) & \mbox{if}\ 1/p+1/r<1;
\end{cases}
\end{align*}
and $E_{p,q,r}(k)$ is decreasing on $[0,1]$ and
$$E_{p,q,r}(0)=\frac{\pi_{p,q}}{2},\quad
E_{p,q,r}(1)
=\frac{\pi_{v,q}}{2}\ (1/v=1/p-1/r^*).$$
%\frac{1}{q}B\left(\frac{1}{q},\frac{1}{p^*}+\frac{1}{r^*}\right).$$
\end{prop}

For $p \in \mathbb{P}^*$ and $q \in (1,\infty)$,
the \textit{generalized trigonometric function} $\sin_{p,q}{x}$ is the inverse function of
$$\sin_{p,q}^{-1}{x}:=
\begin{cases}
\disp \int_0^x \frac{dt}{(1-t^q)^{1/p}} & \mbox{if}\ p \neq \infty,\\
x & \mbox{if}\ p=\infty.
\end{cases}
$$
Clearly, $\sin_{p,q}{x}$ is increasing function from $[0,\pi_{p,q}/2]$ onto $[0,1]$.
%,
%where
%$$\pi_{p,q}:=2\sin_{p,q}^{-1}{1}
%=2 \displaystyle \int_0^1 \dfrac{dt}{(1-t^q)^{1/p}}
%=\frac2q B\left(\frac{1}{q},\frac{1}{p^*}\right).$$
%We also define $\pi_{p,q}$ for negative $p$ by the right-hand side of 
%the equation above. Moreover, $\pi_{\pm \infty,q}:=2$.

For $p=q=2$, $\sin_{p,q}{\theta}$ and $\pi_{p,q}=2\sin_{p,q}^{-1}{1}$ are 
identical to the classical $\sin{\theta}$ and $\pi$, respectively.
Moreover, $\sin_{p,q}{\theta}$ and $\pi_{p,q}$ 
play important roles to express the solutions $(\lambda,u)$ 
of inhomogeneous eigenvalue problem of $p$-Laplacian
$-(|u'|^{p-2}u')'=\lambda |u|^{q-2}u,\ p,\ q \in (1,\infty)$,
with a boundary condition (see \cite{DM,LE,T}
and the references given there).

For $p \neq \infty$ and $x \in (0,\pi_{p,q}/2)$, we also define 
$\cos_{p,q}{x}:=(\sin_{p,q}{x})'$.
It is easy to check that for $x \in (0,\pi_{p,q}/2)$,
\begin{gather*}
\cos_{p,q}^p{x}+\sin_{p,q}^q{x}=1,\quad 
%(\sin_{p,q}{x})'=\cos_{p,q}{x},\notag \\
(\cos_{p,q}{x})'=-\frac{q}{p}\sin_{p,q}^{q-1}{x}\cos_{p,q}^{2-p}{x}.
%(\cos_{p,q}^{p-1}{x})'=-\frac{q}{p^*}\sin_{p,q}^{q-1}{x}.
\end{gather*}

Now, we apply the generalized trigonometric function to the generalized
complete elliptic integrals.
For $p \in \mathbb{P}^*$ and $q,\ r \in (1,\infty)$, 
using $\sin_{p,q}{\theta}$ and $\pi_{p,q}$,  
we can express $K_{p,q,r}(k)$ and $E_{p,q,r}(k)$ 
as follows.
\begin{align*}
K_{p,q,r}(k)&=\int_0^{\pi_{p,q}/2}
\frac{d\theta}{(1-k^q\sin_{p,q}^{q}{\theta})^{1/r}},\\
E_{p,q,r}(k)&=\int_0^{\pi_{p,q}/2}
(1-k^q\sin_{p,q}^{q}{\theta})^{1/r^*}\,d\theta.
\end{align*}

Then, we see that 
the functions $K_{p,q,r}(k)$ and $E_{p,q,r}(k)$ satisfy a system of linear 
differential equations.
\begin{prop}
\label{prop:p-differential}
Let $p \in \mathbb{P}^*,\ q,\ r \in (1,\infty)$. Then,
\begin{align*}
\frac{dE_{p,q,r}}{dk}&=\frac{q(E_{p,q,r}-K_{p,q,r})}{r^*k},\\
\frac{dK_{p,q,r}}{dk}&=\dfrac{aE_{p,q,r}-(a-k^q)K_{p,q,r}}{k(1-k^q)},
\end{align*}
where $a:=1+q/r^*-q/p$.
\end{prop}

\begin{proof}
We consider the case $p \neq \infty$.
Differentiating $E_{p,q,r}(k)$ we have
\begin{align*}
\frac{dE_{p,q,r}}{dk}
%&=\int_0^{\pi_{p,q}/2}
%\frac{d}{dk}(1-k^q\sin_{p,q}^q{\theta})^{1/r^*}\,d\theta\\
&=\frac{q}{r^*}\int_0^{\pi_{p,q}/2}
\dfrac{-k^{q-1}\sin_{p,q}^q{\theta}}{(1-k^q\sin_{p,q}^q{\theta})
^{1/r}}\,d\theta
%&=\frac{q}{r^*k} \int_0^{\pi_{p,q}/2}
%\dfrac{1-k^q\sin_{p,q}^q{\theta}-1}{(1-k^q\sin_{p,q}^q{\theta})^{1/r}}
%\,d\theta\\
=\frac{q}{r^*k} (E_{p,q,r}-K_{p,q,r}).
\end{align*}

Next, for $K_{p,q,r}(k)$ 
\begin{align*}
\frac{dK_{p,q,r}}{dk}
=\frac{q}{r}\int_0^{\pi_{p,q}/2}
\dfrac{k^{q-1}\sin_{p,q}^q{\theta}}
{(1-k^q\sin_{p,q}^q{\theta})^{1+1/r}}\,d\theta.
\end{align*}
Here we see that 
\begin{align*}
& \frac{d}{d\theta} 
\left(\frac{-\cos_{p,q}^{p/r}{\theta}}
{(1-k^q\sin_{p,q}^q{\theta})^{1/r}}\right)
%&=\frac{(q/r)\sin_{p,q}^{q-1}{\theta}\cos_{p,q}^{1-p/r^*}{\theta}
%(1-k^q\sin_{p,q}^q{\theta})
%-(q/r)k^q\sin_{p,q}^{q-1}{\theta}\cos_{p,q}^{p/r+1}{\theta}}
%{(1-k^q\sin_{p,q}^q{\theta})^{1+1/r}}\\
=\frac{q(1-k^q)\sin_{p,q}^{q-1}{\theta}\cos_{p,q}^{1-p/r^*}{\theta}}
{r(1-k^q\sin_{p,q}^q{\theta})^{1+1/r}},\\
& \lim_{\theta \to \pi_{p,q}/2} \cos_{p,q}^{p-1}{\theta}
=\lim_{\theta \to \pi_{p,q}/2} (1-\sin_{p,q}^q{\theta})^{1/p^*}=0;
\end{align*}
so that we use integration by parts as
\begin{align*}
\frac{dK_{p,q,r}}{dk}
&=\frac{k^{q-1}}{1-k^q}\int_0^{\pi_{p,q}/2}
\frac{d}{d\theta} 
\left(\frac{-\cos_{p,q}^{p/r}{\theta}}
{(1-k^q\sin_{p,q}^q{\theta})^{1/r}}\right)
\sin_{p,q}{\theta}\cos_{p,q}^{p/r^*-1}{\theta}\,d\theta\\
&=\frac{k^{q-1}}{1-k^q} \left[\frac{-\sin_{p,q}{\theta}\cos_{p,q}^{p-1}{\theta}}
{(1-k^q\sin_{p,q}^q{\theta})^{1/r}}
\right]_0^{\pi_{p,q}/2}\\
& \qquad
+\frac{k^{q-1}}{1-k^q} \int_0^{\pi_{p,q}/2} 
\frac{\cos_{p,q}^{p/r}{\theta}}{(1-k^q\sin_{p,q}^q{\theta})^{1/r}}
\left(\cos_{p,q}^{p/r^*}{\theta}-
%q\left(\frac{1}{r}-\frac{1}{p}\right)
\frac{(q/r^*-q/p)\sin_{p,q}^q{\theta}}{\cos_{p,q}^{p/r}{\theta}}\right)\,d\theta\\
&=\frac{k^{q-1}}{1-k^q} \int_0^{\pi_{p,q}/2} 
\frac{\cos_{p,q}^p{\theta}-(q/r^*-q/p)\sin_{p,q}^q{\theta}}
{(1-k^q\sin_{p,q}^q{\theta})^{1/r}}\,d\theta\\
&=\frac{k^{q-1}}{1-k^q} \int_0^{\pi_{p,q}/2}
\frac{(1+q/r^*-q/p)(1-k^q\sin_{p,q}^q{\theta})-(1+q/r^*-q/p-k^q)}{k^q(1-k^q\sin_{p,q}^q{\theta})^{1/r}}\,d\theta\\
&=\frac{(1+q/r^*-q/p)E_{p,q,r}-(1+q/r^*-q/p-k^q)K_{p,q,r}}{k(1-k^q)}.
\end{align*}

The case $p=\infty$ is proved similarly. Indeed,
\begin{align*}
\frac{dE_{\infty,q,r}}{dk}
%&=\int_0^{\pi_{p,q}/2}
%\frac{d}{dk}(1-k^q\sin_{p,q}^q{\theta})^{1/r^*}\,d\theta\\
&=\frac{q}{r^*}\int_0^1
\dfrac{-k^{q-1}\theta^q}{(1-k^q\theta^q)
^{1/r}}\,d\theta
%&=\frac{q}{r^*k} \int_0^{\pi_{p,q}/2}
%\dfrac{1-k^q\sin_{p,q}^q{\theta}-1}{(1-k^q\sin_{p,q}^q{\theta})^{1/r}}
%\,d\theta\\
=\frac{q}{r^*k} (E_{\infty,q,r}-K_{\infty,q,r})
\end{align*}
and
\begin{align*}
\frac{dK_{\infty,q,r}}{dk}
&=\frac{q}{r}\int_0^1
\dfrac{k^{q-1}\theta^q}
{(1-k^q\theta^q)^{1+1/r}}\,d\theta\\
%\end{align*}
%Here we see that 
%\begin{align*}
%& \frac{d}{d\theta} 
%\left(\frac{-\cos_{p,q}^{p/r^*}{\theta}}
%{(1-k^q\sin_{p,q}^q{\theta})^{1/r^*}}\right)
%&=\frac{(q/r)\sin_{p,q}^{q-1}{\theta}\cos_{p,q}^{1-p/r^*}{\theta}
%(1-k^q\sin_{p,q}^q{\theta})
%-(q/r)k^q\sin_{p,q}^{q-1}{\theta}\cos_{p,q}^{p/r+1}{\theta}}
%{(1-k^q\sin_{p,q}^q{\theta})^{1+1/r}}\\
%=\frac{q(1-k^q)\sin_{p,q}^{q-1}{\theta}\cos_{p,q}^{1-p/r}{\theta}}
%{r^*(1-k^q\sin_{p,q}^q{\theta})^{1+1/r^*}},\\
%& \lim_{\theta \to \pi_{p,q}/2} \cos_{p,q}^{p-1}{\theta}
%=\lim_{\theta \to \pi_{p,q}/2} (1-\sin_{p,q}^q{\theta})^{1/p^*}=0;
%\end{align*}
%so that we use integration by parts as
%\begin{align*}
%\frac{dK_{p,q,r}}{dk}
&=\frac{k^{q-1}}{1-k^q}\int_0^1
\frac{d}{d\theta} 
\left(-\left(\frac{1-\theta^q}{1-k^q\theta^q}\right)^{1/r}\right)
\theta(1-\theta^q)^{1/r^*}\,d\theta\\
&=\frac{k^{q-1}}{1-k^q} \left[\frac{-\theta(1-\theta^q)}
{(1-k^q\theta^q)^{1/r}}\right]_0^1\\
& \qquad
+\frac{k^{q-1}}{1-k^q} \int_0^1 
\left(\frac{1-\theta^q}{1-k^q\theta^q}\right)^{1/r}
\left((1-\theta^q)^{1/r^*}-\frac{(q/r)\theta^q}{(1-\theta^q)^{1/r}}\right)\,d\theta\\
&=\frac{k^{q-1}}{1-k^q} \int_0^1 
\frac{1-\theta^q-(q/r)\theta^q}
{(1-k^q\theta^q)^{1/r}}\,d\theta\\
&=\frac{k^{q-1}}{1-k^q} \int_0^1
\frac{(1+q/r^*)(1-k^q\theta^q)-(1+q/r^*-k^q)}{k^q(1-k^q\theta^q)^{1/r}}\,d\theta\\
&=\frac{(1+q/r^*)E_{\infty,q,r}-(1+q/r^*-k^q)K_{\infty,q,r}}{k(1-k^q)}.
\end{align*}
This completes the proof.
\end{proof}

%%%%%

Proposition \ref{prop:p-differential} now yields 
Theorem \ref{thm:pqrL}.

\begin{proof}[Proof of Theorem \ref{thm:pqrL}]
Let $k':=(1-k^q)^{1/r},\ E'_{p,r,q^*}(k):=E_{p,r,q^*}(k')$ and 
$K'_{p,r,q^*}(k):=K_{p,r,q^*}(k')$.
As $dk'/dk=-(q/r)k^{q-1}/(k')^{r-1}$,
Proposition \ref{prop:p-differential} gives
\begin{align*}
\frac{dE_{p,q,r^*}}{dk}&=\frac{q(E_{p,q,r^*}-K_{p,q,r^*})}{rk},\\
\frac{dK_{p,q,r^*}}{dk}&=\dfrac{aE_{p,q,r^*}-(a-k^q)K_{p,q,r^*}}{k(k')^r},\\
\frac{dE'_{p,r,q^*}}{dk}&=\frac{k^{q-1}(-E'_{p,r,q^*}+K'_{p,r,q^*})}{(k')^r},\\
\frac{dK'_{p,r,q^*}}{dk}&=\frac{q(-bE'_{p,r,q^*}+(b-(k')^r)K'_{p,r,q^*})}{rk(k')^r},
\end{align*}
where $a:=1+q/r-q/p$ and $b:=1+r/q-r/p$.

We denote the left-hand side of \eqref{eq:pqrL} by $L(k)$.
A direct computation shows that
\begin{align*}
\frac{d}{dk} & L(k)\\
&  = 
\frac{q(E_{p,q,r^*}-K_{p,q,r^*})}{rk} \cdot K'_{p,r,q^*}
+E_{p,q,r^*} \cdot \frac{q(-bE'_{p,r,q^*}+(b-(k')^r)K'_{p,r,q^*})}{rk(k')^r}\\
& \quad 
+\dfrac{aE_{p,q,r^*}-(a-k^q)K_{p,q,r^*}}{k(k')^r} \cdot E'_{p,r,q^*}
+K_{p,q,r^*} \cdot \frac{k^{q-1}(-E'_{p,r,q^*}+K'_{p,r,q^*})}{(k')^r} \\
& \qquad 
-\dfrac{aE_{p,q,r^*}-(a-k^q)K_{p,q,r^*}}{k(k')^r} \cdot K'_{p,r,q^*}
-K_{p,q,r^*} \cdot \frac{q(-bE'_{p,r,q^*}+(b-(k')^r)K'_{p,r,q^*})}{rk(k')^r}\\
& =
\left(\frac{q}{rk}+\frac{q(b-(k')^r)}{rk(k')^r}-\frac{a}{k(k')^r}\right)E_{p,q,r^*}K'_{p,r,q^*}\\
& \quad
+\left(-\frac{q}{rk}+\frac{k^{q-1}}{(k')^r}+\frac{a-k^q}{k(k')^r}-\frac{q(b-(k')^r)}{rk(k')^r}\right)K_{p,q,r^*}K'_{p,r,q^*}\\
& \qquad
+\left(-\frac{qb}{rk(k')^r}+\frac{a}{k(k')^r}\right)E_{p,q,r^*}E'_{p,r,q^*}\\
& \qquad \quad
+\left(-\frac{a-k^q}{k(k')^r}-\frac{k^{q-1}}{(k')^r}+\frac{qb}{rk(k')^r}\right)K_{p,q,r^*}E'_{p,r,q^*}\\
& =
\frac{qb-ra}{rk(k')^r}(E_{p,q,r^*}K'_{p,r,q^*}-K_{p,q,r^*}K'_{p,r,q^*}-E_{p,q,r^*}E'_{p,r,q^*}+K_{p,q,r^*}E'_{p,r,q^*}).
\end{align*}
Since $qb-ra=0$, we see that $dL/dk=0$.
Thus $L(k)$ is a constant $C$.

We will evaluate $C$ as follows. 
Since
\begin{align*}
|(K_{p,q,r^*} & -E_{p,q,r^*})K'_{p,r,q^*}| \\
&=\int_0^{\pi_{p,q}/2} 
\left(\frac{1}{(1-k^q\sin_{p,q}^q{\theta})^{1/r^*}}
-(1-k^q\sin_{p,q}^q{\theta})^{1/r}\right)\,d\theta\\
& \qquad \times \int_0^{\pi_{p,r}/2} 
\frac{d\theta}{(1-(k')^r\sin_{p,r}^r{\theta})^{1/q^*}}\\
&=\int_0^{\pi_{p,q}/2} 
\frac{k^q\sin_{p,q}^q{\theta}}{(1-k^q\sin_{p,q}^q{\theta})^{1/r^*}}\,d\theta \cdot
\int_0^{\pi_{p,r}/2} 
\frac{d\theta}{(\cos_{p,r}^p{\theta}+k^q\sin_{p,r}^r{\theta})^{1/q^*}}\\
& \leq k^q K_{p,q,r^*}(k) \cdot \frac{1}{k^{q-1}}\frac{\pi_{p,r}}{2}\\
& =\frac{\pi_{p,r}}{2}kK_{p,q,r^*}(k),
\end{align*}
we obtain $\lim_{k \to +0}(K_{p,q,r^*}-E_{p,q,r^*})K'_{p,r,q^*}=0$.
Therefore, from Proposition \ref{prop:behaviour} 
$$C=\lim_{k \to +0}K_{p,q,r^*}E'_{p,r,q^*}
=K_{p,q,r^*}(0)E_{p,r,q^*}(1)
=\frac{\pi_{p,q}\pi_{s,r}}{4},$$
%4r}B\left(\frac{1}{r},\frac{1}{p^*}+\frac{1}{q}\right).$$
where $1/s=1/p-1/q$. Thus, we conclude the assertion.
\end{proof}

Finally, we will give a remark for Theorem \ref{thm:pqrL}. 
From the series expansion and the termwise integration, 
it is possible to express the generalized complete elliptic integrals
by Gaussian hypergeometric functions
\begin{align*}
K_{p,q,r}(k)&=\frac{\pi_{p,q}}{2}F\left(\frac{1}{q},\frac{1}{r};\frac{1}{p^*}+\frac{1}{q};k^q\right),\\
E_{p,q,r}(k)&=\frac{\pi_{p,q}}{2}F\left(\frac{1}{q},-\frac{1}{r^*};\frac{1}{p^*}+\frac{1}{q};k^q\right).
\end{align*}
By these expressions and letting $1/p=1/2-b,\ 1/q=1/2+a,\ 1/r=1/2-c$ and $k^q=x$
in \eqref{eq:pqrL}, we obtain \textit{Elliott's identity}
(see Elliott \cite{El}; see also \cite{AQVV}, \cite[Theorem 3.2.8]{AAR} and 
\cite[(13) p.\,85]{EMOT}):
\begin{multline}
\label{eq:elliott}
F\left({{1/2+a,-1/2-c}\atop{a+b+1}} ;x\right)
F\left({{1/2-a,1/2+c}\atop{b+c+1}} ;1-x\right)\\
+F\left({{1/2+a,1/2-c}\atop{a+b+1}};x\right)
F\left({{-1/2-a,1/2+c}\atop{b+c+1}};1-x\right)\\
-F\left({{1/2+a,1/2-c}\atop{a+b+1}};x\right)
F\left({{1/2-a,1/2+c}\atop{b+c+1}};1-x\right)\\
=\frac{\Gamma(a+b+1)\Gamma(b+c+1)}{\Gamma(a+b+c+3/2)\Gamma(b+1/2)}
\end{multline}
for $|a|,\,|c|<1/2$ and $b \in (-1/2,\infty)$,
where $\Gamma$ denotes the gamma function.
Also, letting $1/p=2-c-a$ and $1/q=1-a$ in \eqref{eq:pqL} of 
Corollary \ref{cor:q=r}, we have the identity of \cite[Corollary 3.13 (5)]{AQVV}
for $a \in (0,1)$ and $c \in (1-a,\infty)$.
A series of Vuorinen's works on Elliott's identity with his coauthors 
starting from \cite{AQVV}
deals with the concavity/convexity properties of certain related functions
to the left-hand side of \eqref{eq:elliott}.

%For $a=b=c=0$, \eqref{eq:elliott} coincides with Legendre's relation \eqref{eq:L}.

%\begin{multline}
%\label{eq:elliott}
%F\left({{\frac12+a,-\frac12-c}\atop{a+b+1}} ;x\right)
%F\left({{\frac12-a,c+\frac12}\atop{b+c+1}} ;1-x\right)\\
%+F\left({{a+\frac12,\frac12-c}\atop{a+b+1}};x\right)
%F\left({{-(a+\frac12),c+\frac12}\atop{b+c+1}};1-x\right)\\
%-F\left({{a+\frac12,\frac12-c}\atop{a+b+1}};x\right)
%F\left({{\frac12-a,c+\frac12}\atop{b+c+1}};1-x\right)\\
%=\frac{\Gamma(a+b+1)\Gamma(b+c+1)}{\Gamma(a+b+c+\frac32)\Gamma(b+\frac12)},
%\end{multline}
%Of course, from Proposition \ref{prop:hypergeometricexpression} and Elliott's
%identity with $a=-b=c=1/p-1/2$ and $x=k^p$,
%we are also able to obtain Theorem \ref{thm:p-legendre}.

%%%%

% \newpage

% Abbreviated version of the title for the running head:

% ``Generalized elliptic functions and its applications to
% the $p$-Laplacian''
% \bigskip\\

% Name and mailing address of the author to whom proofs should be sent:

% Shingo Takeuchi

% shingo@cc.kogakuin.ac.jp

\end{document}